\documentclass[10pt,a4paper]{article}

\usepackage{a4wide}
\usepackage{amsthm}
\usepackage{amsfonts}
\usepackage{amsmath}
\usepackage[english]{babel}

\usepackage{faktor}
\usepackage{etex}
\usepackage{hyperref}
\usepackage{multimedia}
\usepackage[all]{xy}
\usepackage{amssymb}
\usepackage{graphicx,textpos}
\usepackage[dvipsnames]{xcolor}
\usepackage{helvet}
\usepackage{stmaryrd}
\usepackage{enumerate}
\usepackage{todonotes}
\usepackage{pdfsync}
\usepackage{dsfont}
\usepackage[shortcuts]{extdash}
\usepackage[all]{xy}
  \newdir{ >}{{}*!/-9pt/@{>}}

\newcommand{\N}{\mathbb{N}}
\newcommand{\R}{\mathbb{R}}
\newcommand{\Q}{\mathbb{Q}}
\newcommand{\Z}{\mathbb{Z}}

\newcommand{\I}{\mathds{1}}

\newcommand{\rst}[1]{\ensuremath{{\mathbin\mid}\raise-.5ex\hbox{$#1$}}}

\newcommand{\lien}{\mathfrak{n}}

\DeclareMathOperator{\GL}{GL}

\DeclareMathOperator{\Aut}{Aut}

\author{Jonas Der\'e\thanks{Email: \href{mailto:jonas.dere@kuleuven.be}{jonas.dere@kuleuven.be} The author was supported by a postdoctoral fellowship of the Research Foundation -- Flanders (FWO).}}
\title{\textbf{Strongly scale-invariant \\ virtually polycyclic groups}}
\date{
}

\newtheorem{Def}{Definition}[section]
\newtheorem{Ex}[Def]{Example}
\newtheorem{Cor}[Def]{Corollary}
\newtheorem{Thm}[Def]{Theorem}
\newtheorem{Prop}[Def]{Proposition}
\newtheorem{Lem}[Def]{Lemma}

\newtheorem*{Prop*}{Proposition}
\newtheorem*{Lem*}{Lemma}

\newtheorem{Con}{Conjecture}
\makeatletter
\newtheorem*{rep@theorem}{\rep@title}
\newcommand{\newreptheorem}[2]{%
	\newenvironment{rep#1}[1]{%
		\def\rep@title{#2 \ref{##1}}%
		\begin{rep@theorem}}%
		{\end{rep@theorem}}}
\makeatother

\newreptheorem{Thm}{Theorem}
\newreptheorem{Cor}{Corollary}

\newcommand{\suchthat}{\;\ifnum\currentgrouptype=16 \middle\fi|\;}

\newcommand\restr[2]{{
		\left.\kern-\nulldelimiterspace 
		#1 
		\vphantom{\big|} 
		\right|_{#2} 
}}

\begin{document}

\maketitle

\begin{abstract}
	A finitely generated group $\Gamma$ is called strongly scale-invariant if there exists an injective endomorphism $\varphi: \Gamma \to \Gamma$ with the image $\varphi(\Gamma)$ of finite index in $\Gamma$ and the subgroup $\displaystyle \bigcap_{n>0}\varphi^n(\Gamma)$ finite. The only known examples of such groups are virtually nilpotent, or equivalently, all examples have polynomial growth. A question by Nekrashevych and Pete asks whether these groups are the only possibilities for such endomorphisms, motivated by the positive answer due to Gromov in the special case of expanding group morphisms. 
	
	In this paper, we study this question for the class of virtually polycyclic groups, i.e.~the virtually solvable groups for which every subgroup is finitely generated. Using the $\Q$-algebraic hull, which allows us to extend the injective endomorphisms of certain virtually polycyclic groups to a linear algebraic group, we show that the existence of such an endomorphism implies that the group is virtually nilpotent. Moreover, we fully characterize which virtually nilpotent groups have a morphism satisfying the condition above, related to the existence of a positive grading on the corresponding radicable nilpotent group. As another application of the methods, we generalize a result of Fel'shtyn and Lee about which maps on infra-solvmanifolds can have finite Reidemeister number for all iterates.
\end{abstract}

\section{Introduction}

Expanding maps were introduced by M.~Shub in \cite{shub69-1} as one of the first examples in dynamical systems combining structural stability with chaos. The natural question which closed manifolds admit an expanding map was only answered after more than ten years by M.~Gromov in \cite{grom81-1}, by showing that such a manifold must be an infra-nilmanifold, i.e.~a compact quotient of a simply connected nilpotent Lie group by a discrete group of isometries. Although many infra-nilmanifolds admit expanding maps, for example all infra-nilmanifolds modeled on a nilpotent Lie group of nilpotency class $\leq 2$, there are some examples which do not have expanding maps, for instance the ones modeled on a characteristically nilpotent Lie algebra as in  \cite{dl57-1}. Only much later, a full characterization of the infra-nilmanifolds admitting an expanding map was given in \cite{dd14-1}, by showing that the existence of an expanding map depends only on the covering nilpotent Lie group and is equivalent to the existence of a positive grading on the corresponding Lie algebra. 

On the level of the fundamental group, every expanding map induces an injective endomorphism $\varphi:\Gamma \to \Gamma$ with image $\varphi(\Gamma)$ of finite index in $\Gamma$ and such that the subgroup $\displaystyle \bigcap_{n>0}\varphi^n(\Gamma)$ is finite. In short, we will call an endomorphism $\varphi: \Gamma \to \Gamma$ satisfying the latter conditions strongly scale-invariant. Motivated by the result in \cite{grom81-1}, it is conjectured that the only finitely generated groups admitting a strongly scale-invariant endomorphism are the virtually nilpotent groups, see \cite{np11-1}. This conjecture is hence the equivalent of the theorem by M.~Gromov for morphisms of finitely generated groups.

\begin{Con}
	\label{con:main}
	Every finitely generated strongly scale-invariant group is virtually nilpotent.
\end{Con}

\noindent In the paper \cite{np11-1}, the authors first disprove a weaker conjecture by I.~Benjamini, where one only assumes a nested sequence of finite index subgroups $\Gamma_n <\Gamma$ with $\displaystyle \bigcap_{n>0} \Gamma_n$ finite and the groups $\Gamma_n$ isomorphic to the original group $\Gamma$. The counterexamples, constructed via a semi-direct product of groups with certain properties, do not have a single endomorphism $\varphi:\Gamma \to \Gamma$ realizing the nested sequence as $\varphi^n(\Gamma)= \Gamma_n$ and hence Conjecture \ref{con:main} remains open. 

So far, Conjecture \ref{con:main} has only be studied for some special cases. Firstly, if the images $\varphi^n(\Gamma)$ are normal subgroups of $\Gamma$, then \cite[Theorem 1.1.]{vanl18-1} shows that $\Gamma$ is virtually abelian. On the other hand, with additional assumptions on a profinite group associated to $\varphi$, \cite[Corollary 1.5.]{hlv20-1} implies that $\Gamma$ must be virtually nilpotent. In this paper, we prove Conjecture \ref{con:main} in the special case of virtually polycyclic groups. Note that, as a consequence of the existence of the Hirsch length on these groups, every injective group morphism automatically has an image of finite index. 
\begin{repThm}{thm:main}
	Let $\Gamma$ be a virtually polycyclic group which is strongly scale-invariant, then $\Gamma$ is virtually nilpotent.
\end{repThm}
\noindent For the proof, we first show how to extend an injective endomorphism $\varphi$ to the $\Q$-algebraic hull of a injectively characteristic finite index subgroup. By using the structure of solvable linear algebraic groups, we deduce that some iterate of $\varphi$ must leave an infinite subgroup invariant if the group $\Gamma$ is not virtually nilpotent.

To every endomorphism $\varphi: \Gamma \to \Gamma$ of any group $\Gamma$, we can associate an invariant $R(\varphi) \in \N \cup \{\infty\}$, called the Reidemeister number. The exact definition will be given in Section \ref{sec:reidemeister}. By considering the iterates $\varphi^n: \Gamma \to \Gamma$, we can make a zeta function from the Reidemeister numbers $R(\varphi^n)$, as long as they are all different from $\infty$. As a consequence of the proof, we show that the only virtually polycyclic groups which have an injective endomorphism $\varphi$ with a well-defined Reidemeister zeta function are the virtually nilpotent ones. 
\begin{repThm}{cor:felshtyn}
		Let $\varphi: \Gamma \to \Gamma$ be a monomorphism of a virtually polycyclic group. If $R(\varphi^n) < \infty$ for all integers $n > 0$ the group $\Gamma$ must be virtually nilpotent. In particular the Reidemeister zeta function $R_\varphi(z)$ of $\varphi$ is rational if the group $\Gamma$ is additionally torsion-free.
\end{repThm}
\noindent This generalizes the result in \cite{fl13-2} which only deals with automorphisms on a restrictive class of virtually polycyclic groups, namely lattices in a Lie group of type $(R)$. It is also related to the conjecture in \cite{fh94-1} that every finitely generated group admitting an injective endomorphism with finite Reidemeister number must have subexponential growth.

Conjecture \ref{con:main} and Theorem \ref{thm:main} raise the problem of characterizing the virtually nilpotent groups which are strongly scale-invariant. To every finitely generated virtually nilpotent group $\Gamma$, we can associate a unique radicable group $N^\Q$. By generalizing the methods of \cite{dd14-1}, which only deal with fundamental groups of infra-nilmanifolds and thus only torsion-free groups, we give the following theorem, showing that it suffices to know $N^\Q$ to decide whether or not the group $\Gamma$ is strongly scale-invariant.

\begin{repThm}{thm:virtuallynilpotent}
Let $\Gamma$ be a finitely generated virtually nilpotent group with associated radicable nilpotent group $N^\Q$. The group $\Gamma$ is strongly scale-invariant if and only if the Lie algebra corresponding to $N^\Q$ has a positive grading. Moreover, if $\Gamma$ is strongly scale-invariant, there exists a monomorphism $\varphi: \Gamma \to \Gamma$ such that $\displaystyle \bigcap_{n>0} \varphi^n(\Gamma)$ is the maximal finite normal subgroup of $\Gamma$.
\end{repThm}
\noindent For torsion-free nilpotent groups, this was studied in \cite{corn14-1}, whereas the more general case of torsion-free virtually nilpotent groups was given in \cite{dd14-1}. Theorem \ref{thm:virtuallynilpotent} in particular shows that being strongly scale-invariant for virtually nilpotent groups is preserved under taking finite index subgroups. It is conjectured that two finitely generated torsion-free nilpotent groups $N_1$ and $N_2$ are quasi-isometric if and only if the associated real Lie groups $N_1^\R$ and $N_2^\R$, which contain $N_1$ and $N_2$ as cocompact lattices, are isomorphic. Hence this result motivates the conjecture that being strongly scale-invariant is a quasi-isometric invariant for virtually nilpotent groups, or even for general finitely generated groups if Conjecture \ref{con:main} is true.

We start by giving some preliminaries about virtually polycyclic group, their $\Q$-algebraic hulls and the Mal'cev completion of nilpotent groups in Section \ref{sec:prel}. Afterwards, Section \ref{sec:generalprop} states general properties of strongly scale-invariant endomorphisms, allowing us to simply the situation for virtually polycyclic groups. Next, we prove Theorem \ref{thm:main} in Section \ref{sec:maintheorem} and conclude in Section \ref{sec:reidemeister} with the discussion of finite Reidemeister numbers. Finally we characterize the virtually nilpotent groups which are strongly scale-invariant in Section \ref{sec:virtuallynilpotent}.

\section{Preliminaries}
\label{sec:prel}
In this section we introduce terminology for the rest of the paper and recall some properties for both nilpotent and virtually polycyclic groups. For more background and further details we refer to \cite{sega83-1}. All groups we consider are assumed to be finitely generated. 
\paragraph*{Notation} 
The following definition allows us to simplify some statements in the paper.
\begin{Def}
	\label{def:ssi}
	Let $\Gamma$ be any group with endomorphism $\varphi:\Gamma \to \Gamma$.
	\begin{itemize}
		\item We call $\varphi: \Gamma \to \Gamma$ a \textbf{monomorphism} of $\Gamma$ if it is injective.
		\item A monomorphism $\varphi: \Gamma \to \Gamma$ is called \textbf{strongly scale-invariant} if $\varphi(\Gamma)$ has finite index in $\Gamma$ and $\displaystyle \bigcap_{n > 0} \varphi^n(\Gamma)$ is finite.
		\item The group $\Gamma$ is called \textbf{strongly scale-invariant} if it admits a monomorphism $\varphi: \Gamma \to \Gamma$ which is strongly scale-invariant.	
	\end{itemize}
\end{Def}
\noindent Let $\Gamma$ be a group with a monomorphism $\varphi: \Gamma \to \Gamma$, then we say that a subgroup $\Gamma^\prime < \Gamma$ is $\varphi$-invariant if $\varphi(\Gamma^\prime) < \Gamma^\prime$. If $\Gamma^\prime$ is $\varphi$-invariant for all monomorphisms $\varphi: \Gamma \to \Gamma$ then we call $\Gamma^\prime$ an \emph{injectively characteristic subgroup} of $\Gamma$. Note that most injectively characteristic subgroups we consider will in fact be fully characteristic, meaning it is invariant under all endomorphisms of the group $\Gamma$.
Every injectively characteristic subgroup is a normal subgroup, since it is invariant under all conjugation maps.

\begin{Ex}
\noindent For any group $\Gamma$ and every integer $k > 0$, we denote by $\Gamma^k$ the subgroup generated by all the $k$-th powers $\gamma^k$ of elements $\gamma \in \Gamma$. The subgroups $\Gamma^k <\Gamma$ are injectively characteristic, and for every subgroup $\Gamma^\prime$ of finite index in $\Gamma$, there exists an integer $k > 0$ such that $\Gamma^k < \Gamma^\prime$.
\end{Ex}

\paragraph{Rational Mal'cev completion}

For every group $N$, we can define the lower central series $\gamma_i(N)$ inductively via $\gamma_1(N) = N$ and $\gamma_{i+1}(N) = [N,\gamma_i(N)]$. A group $N$ is called $c$-step nilpotent if $\gamma_{c+1}(N) = \{e\}$ and $\gamma_c(N) \neq \{e\}$. The isolator of a subgroup $H < N$ in $N$ is defined as $$\sqrt[N]{H} = \{x \in N \suchthat \exists m > 0 : x^m \in H \}.$$ The isolators $\sqrt[N]{\gamma_i(N)}$ of the lower central series of a finitely generated torsion-free nilpotent group $N$ are injectively characteristic subgroups such that the quotients $\faktor{\sqrt[N]{\gamma_i(N)}}{ \sqrt[N]{\gamma_{i+1}(N)}}$ are torsion-free and hence isomorphic to $\Z^{m_i}$ for some $m_i \geq 0$.

Let $N$ be a finitely generated torsion-free nilpotent group, then $N$ embeds as a subgroup of a radicable torsion-free nilpotent group $N^\Q$ such that every element in $N^\Q$ has a power which lies in $N$, and we call $N^\Q$ the \emph{rational Mal'cev completion} of N. Every monomorphism $\varphi: N \to N$ uniquely extends to an automorphism of $N^\Q$. For every integer $ m > 0 $, every element $x \in N^\Q$ has a unique $m$-th rooth, i.e.~a unique element $y \in N^\Q$ such that $y^m = x$ and we denote this unique element $y$ as $y = x^{\frac{1}{m}}$. Any other finitely generated subgroup $M \subset N^\Q$ such that for every $x \in N^\Q$, there exists $m > 0$ such that $x^m \in M$, is called a \emph{full subgroup} of $N^\Q$. In this case, $N^\Q$ is also the rational Mal'cev completion of $M$. For every $k > 0$, the group $N^k$ is a full subgroup of $N^\Q$. If $M < N^\Q$ is a full subgroup, then $M \cap N$ has finite index in both $N$ and $M$. As a consequence, it is well-known which groups have an isomorphic rational Mal'cev completion.
\begin{Prop}
	\label{prop:samemalcev}
	Let $N_1$ and $N_2$ be two torsion-free nilpotent groups, then $N_1^\Q \approx N_2^\Q$ if and only if the groups are abstractly commensurable.
\end{Prop}
\noindent Here, abstractly commensurable means that there exist subgroups $M_1 < N_1$ and $M_2 < N_2$ of finite index such that $M_1 \approx M_2$. 

The rational Mal'cev completion $N^\Q$ uniquely corresponds to a rational Lie algebra $\lien^\Q$ under the exponential map, which forms a bijection in this case. Under this correspondence, every automorphism $\phi$ of the group $N^\Q$ corresponds to an automorphism $\psi$ of $\lien^\Q$. We define the eigenvalues of $\phi: N^\Q \to N^\Q$ as the eigenvalues of this Lie algebra automorphism $\psi$. A \emph{positive grading} on $\lien^\Q$ is a decomposition of $$\lien^\Q = \bigoplus_{i = 1}^k \lien_i$$ into a direct sum of rational subspaces $\lien_i$ such that for every $i, j >0$ it holds that $[\lien_i,\lien_j] \subset \lien_{i+j}$. In this case we call it a positive grading on the nilpotent group $N^\Q$ as well. We say that an automorphism $\psi \in \Aut(\lien^\Q)$ preserves the positive grading if $\psi(\lien_i) = \lien_i$ for all $i$. 

If $\varphi: N \to N$ is a monomorphism, it also induces monomorphisms $$\varphi_i: \faktor{\sqrt[N]{\gamma_i(N)}}{ \sqrt[N]{\gamma_{i+1}(N)}} \approx \Z^{m_i} \to  \faktor{\sqrt[N]{\gamma_i(N)}}{ \sqrt[N]{\gamma_{i+1}(N)}} \approx \Z^{m_i}$$ on the quotients of the isolators of the lower central series. We define the eigenvalues of $\varphi$ as the collection of all eigenvalues of these maps $\varphi_i$. On the other hand, by extending $\varphi: N \to N$, we also get an automorphism $\phi: N^\Q \to N^\Q$, and the eigenvalues of $\varphi$ are equal the eigenvalues of $\phi$. The following lemma will be useful further in the paper.
\begin{Lem}
	\label{lem:inducedmap}
	Let $\phi: N^\Q \to N^\Q$ be an automorphism for which the induced map $$\overline{\phi}: \faktor{N^\Q}{[N^\Q,N^\Q]} \to \faktor{N^\Q}{[N^\Q,N^\Q]}$$ is the identity map, then the automorphism $\phi$ is unipotent. In particular, if $\phi$ has finite order as well, then it must be the trivial automorphism.
\end{Lem}
\begin{proof}
	From the assumption on $\overline{\phi}$, one can show by induction that the induced maps by $\phi$ on $\faktor{\gamma_{i}(N^\Q)}{\gamma_{i+1}(N^\Q)}$ are trivial. This immediately implies that $\phi$ is unipotent, because the group $N^\Q$ is nilpotent. Since automorphisms of finite order are semisimple, the last part follows as well.
\end{proof}

\paragraph{Virtually polycyclic groups}

We call a group $\Gamma$ polycyclic if it has a subnormal series $$\{e\} = \Gamma_0  \triangleleft \Gamma_1 \triangleleft \ldots \triangleleft \Gamma_l = \Gamma$$ with $\Gamma_{i}$ normal in $\Gamma_{i+1}$ and the quotient $\faktor{\Gamma_{i+1}}{\Gamma_i}$ cyclic. A group $\Gamma$ is called virtually polycyclic if it has a subgroup of finite index which is polycyclic. Note that for virtually polycyclic groups the subgroups $\Gamma^k < \Gamma$ are always of finite index in $\Gamma$ by \cite[Lemma 4.4.]{ragh72-1}. Every finitely generated nilpotent group is automatically polycyclic.

Since cyclic groups are abelian, every (virtually) polycyclic group is (virtually) solvable. In fact, virtually polycyclic groups are exactly the finitely generated virtually solvable groups which satisfy the maximality condition, i.e.~every subgroup is finitely generated. The number of infinite cyclic groups $\faktor{\Gamma_{i+1}}{\Gamma_i}$ of a polycyclic group $\Gamma$ does not depend on the choice of subnormal series. It forms hence an invariant of the group $\Gamma$ which is called \emph{the Hirsch length} $h(\Gamma)$. The Hirsch length of a virtually polycyclic group is defined as the Hirsch length of a polycyclic subgroup of finite index. Note that if $\Gamma^\prime < \Gamma$ is a subgroup and $h(\Gamma^\prime) = h(\Gamma)$, then $\Gamma^\prime$ is a finite index subgroup of $\Gamma$. In particular, the image $\varphi(\Gamma)$ of a monomorphism $\varphi: \Gamma \to \Gamma$ has the same Hirsch length as $\Gamma$ and thus has finite index in $\Gamma$. This shows that the first condition to be a strongly scale-invariant monomorphism in Definition \ref{def:ssi} is automatic for virtually polycyclic groups.

A crucial tool for this paper will be extending monomorphisms to automorphisms of a linear algebraic group. Recall that a (real) \emph{linear algebraic group} defined over a subfield $K \subset \R$ is a subgroup $G \subset \GL(n,\R)$ given by the zeros of a finite number of polynomials over $K$, equipped with the Zariski-topology. An algebraic morphism $G \to G$ defined over $K$ is a group morphism of which the coordinate functions can be written as polynomials over $K$. The unipotent radical is the set of unipotent elements in the radical of $G$, where the radical of $G$ is the maximal connected normal solvable subgroup. Every linear algebraic group can also be considered as a Lie group with a finite number of connected components. The group of $K$-rational points of $G$ is given by the intersection $G(K) = G \cap \GL(n,K)$. For more details about linear algebraic groups we refer to \cite{bore91-1}.

\begin{Def}
	\label{def:hull}
	Let $\Gamma$ be a virtually polycyclic group. A $\Q$-algebraic hull $\Gamma^\Q$ is a linear algebraic group definied over $\Q$ with an injective morphism $i: \Gamma \to \Gamma^\Q$ satisfying the following conditions:
	\begin{enumerate}[$(1)$]
		\item The image $i(\Gamma) < \Gamma^\Q$ is a Zariski-dense subgroup of the $\Q$-rational points of $\Gamma^\Q$;
		\item If $U(\Gamma^\Q)$ is the unipotent radical of $\Gamma^\Q$, then $\dim(U(\Gamma^\Q)) = h(\Gamma)$;
		\item The centralizer of the unipotent radical in $\Gamma^\Q$ is contained in $U(\Gamma^\Q)$, or in symbols $$C_{\Gamma^\Q}(U(\Gamma^\Q)) = Z(U(\Gamma^\Q)),$$ where $Z(U(\Gamma^\Q))$ is the center of $U(\Gamma^\Q)$.
	\end{enumerate}
\end{Def}
\noindent Usually we identify $\Gamma$ with its image under the map $i$. Although the definition is rather technical, the following theorem shows the importance for studying monomorphisms on virtually polycyclic groups, see \cite[Corollary 4.7.]{dere18-1}:

\begin{Thm}
	Let $\varphi: \Gamma \to \Gamma$ be a monomorphism and $\Gamma^\Q$ a $\Q$-algebraic hull of $\Gamma$. There exists a unique algebraic automorphism $\Phi: \Gamma^\Q \to \Gamma^\Q$ defined over $\Q$ such that $\Phi(\gamma) = \varphi(\gamma)$ for all $\gamma \in \Gamma$.
\end{Thm}

\noindent In particular, if a $\Q$-algebraic hull exists, it is unique up to $\Q$-isomorphism of linear algebraic groups, so we will call it the $\Q$-algebraic hull of $\Gamma$.

\begin{Ex}
	If $N$ is a torsion-free nilpotent group, then it embeds as a lattice in a simply connected and connected nilpotent Lie group $N^\R$. The group $N^\R$ has a unique structure as a unipotent linear algebraic group defined over $\Q$. A direct check shows that $N^\R$ forms the $\Q$-algebraic hull of $N$, with the rational Mal'cev completion $N^\Q$ equal to the subgroup of rational points $N^\R(\Q)$ of the algebraic group $N^\R$.
\end{Ex}

\noindent  Every virtually polycyclic group $\Gamma$ has a unique maximal normal nilpotent subgroup, called the Fitting subgroup. If $\Gamma$ has a $\Q$-algebraic hull $\Gamma^\Q$, then its Fitting subgroup $N$ is equal to $N = \Gamma \cap \left(   U(\Gamma^\Q)\right)$ by \cite[Proposition 4.4.]{bg06-1}. On the other hand, every virtually polycyclic group also has a maximal finite normal subgroup $H \triangleleft \Gamma$. By \cite[Corollary 5.10.]{dere18-1} we know that $H=1$ is equivalent to the existence of a $\Q$-algebraic hull.

\begin{Thm}
	\label{thm:existQ}
	A virtually polycyclic group has a $\Q$-algebraic hull if and only if every finite normal subgroup is trivial.
\end{Thm}
\noindent Almost-crystallographic groups are exactly the finitely generated virtually nilpotent groups with only the trivial group as a finite normal subgroup and thus always have a $\Q$-algebraic hull.

\section{Monomorphisms of virtually polycyclic groups}
First we give some general properties of strongly scale-invariant morphisms, then we study the specific situation of virtually polycyclic groups using the $\Q$-algebraic hull and finally we apply the methods in the context of the Reidemeister number.
\subsection{Properties of strongly scale-invariant groups}

\label{sec:generalprop}

In this first part we give some general lemmas, which will allow us to simplify the study of strongly scale-invariant monomorphisms of virtually polycyclic groups. The main goal is to investigate how being strongly scale-invariant behaves under considering iterates of morphisms and under taking finite index subgroups of a group.

The first lemma gives a different characterization of strongly scale-invariant monomorphisms in terms of subgroups which are mapped onto themselves.

\begin{Lem}
	\label{lem:defwithH}
Let $\varphi: \Gamma \to \Gamma$ be a monomorphism with $\varphi(\Gamma)$ of finite index in $\Gamma$, then $\varphi$ is strongly scale-invariant if and only if every subgroup $H < \Gamma$ with $\varphi(H) = H$ is finite.
\end{Lem}
\noindent In other words, a monomorphism with image of finite index is strongly scale-invariant if and only if it can only induce an automorphism on finite subgroups.

\begin{proof}
First assume that $\varphi$ is strongly scale-invariant. Let $H < \Gamma$ be any subgroup satisfying $\varphi(H) = H$, then for every $n > 0$ we have $\varphi^n(H) = H$ as well. In particular, $$H = \displaystyle \bigcap_{n > 0} \varphi^n(H) < \bigcap_{n > 0} \varphi^n(\Gamma)$$ and thus the subgroup $H$ must be finite. 

For the other implication we assume that every subgroup $H <\Gamma$ with $\varphi(H) = H$ is finite. Note that $H_0 = \displaystyle \bigcap_{n > 0} \varphi^n(\Gamma)$ is a subgroup of $\Gamma$ as the intersection of subgroups, hence in order to show that $\varphi$ is strongly scale-invariant, it suffices to prove that $\varphi(H_0) = H_0$. From its definition, it is clear that $\varphi(H_0) < H_0$. For the other inclusion, assume that $x \in H_0$, then we will show that $x \in \varphi(H_0)$. Since $x \in \varphi(\Gamma)$ and $\varphi$ is injective, there exists a unique $y \in \Gamma$ with $x = \varphi(y)$. We have to prove that $y \in H_0$, or thus that $y \in \varphi^n(\Gamma)$ for every $n > 0$. Take $n > 0$ arbitrary, then from $x \in \varphi^{n+1}(\Gamma)$ we know that $x = \varphi^{n+1}(z)$ with $z \in \Gamma$. By the uniqueness of $y$, we know that $y = \varphi^n(z)$ and therefore $y \in \varphi^n(\Gamma)$. We conclude that $H_0 = \varphi(H_0)$, showing that $H$ is finite.
\end{proof}

As an application of the previous lemma, we show that being strongly scale-invariant behaves well under taking iterates of monomorphisms.

\begin{Lem}
	\label{lem:iterate}
Let $\varphi: \Gamma \to \Gamma$ be a monomorphism and $k > 0$ a natural number, then $\varphi$ is strongly scale-invariant if and only if $\varphi^k$ is strongly scale-invariant.
\end{Lem}

\begin{proof}
Note that $ [\Gamma:\varphi^k(\Gamma)] = \displaystyle \prod_{i=1}^k [\varphi^{i-1}(\Gamma):\varphi^{i}(\Gamma)] =\left([\Gamma:\varphi(\Gamma)]\right)^k$ and thus the condition of having an image of finite index is immediate for both implications.

 First consider the reverse implication, namely that $\varphi^k$ strongly scale-invariant implies $\varphi$ strongly scale-invariant. If $H$ is any subgroup with $\varphi(H) = H$, then $\varphi^k(H) = H$ as wel. So, under the assumption that $\varphi^k$ is strongly scale-invariant, it follows that $H$ is finite and thus the lemma follows from Lemma \ref{lem:defwithH}.
	
	For the other implication, assume that $\varphi$ is strongly scale-invariant and that $H$ is a subgroup such that $\varphi^k(H) = H$. Consider the subgroup $\tilde{H}$ generated by the subgroups $H, \varphi(H), \ldots, \varphi^{k-1}(H)$. It is clear that every aforementioned generator of $\tilde{H}$ is the image of another generator under $\varphi$ and thus $\varphi(\tilde{H}) = \tilde{H}$. We conclude that $\tilde{H}$ is finite, hence $H$ as well as  a subgroup of $\tilde{H}$.
\end{proof}

From Definition \ref{def:ssi}, the next observation is obvious.

\begin{Lem}
	\label{lem:restrict}
If $\varphi:\Gamma \to \Gamma$ is a strongly scale-invariant monomorphism of a virtually polycyclic group and $\Gamma^\prime$ is a  $\varphi$-invariant subgroup, then the restriction $\restr{\varphi}{\Gamma^\prime}$ is strongly scale-invariant. 
\end{Lem}
In particular, every injectively characteristic subgroup of a strongly scale-invariant virtually polycyclic group is itself strongly scale-invariant. In the special case of a finite index subgroup in a virtually polycyclic group, the converse of this lemma is true as well. In order to give a proof, we first recall the following fact about finitely generated groups, for which we present the argument for completeness.
\begin{Lem}
Let $\Gamma$ be a finitely generated group and $\varphi: \Gamma \to \Gamma$ an automorphism. For every finite index subgroup $\Gamma^\prime <\Gamma$, there exists a $k$ such that $\varphi^k(\Gamma^\prime) = \Gamma^\prime$.
\end{Lem}
\begin{proof}
Since $\Gamma$ is finitely generated, it only has a finite number of subgroups of index $[\Gamma:\Gamma^\prime]$. For every subgroup $H < \Gamma$ of index $[\Gamma:\Gamma^\prime]$, the image $\varphi(H)$ is again a subgroup of the same index. Therefore $\varphi$ permutes the finite set of subgroups of given index $[\Gamma:\Gamma^\prime]$, showing that some power $\varphi^k$ maps $\Gamma^\prime$ onto itself.
\end{proof}

We are now ready to prove the converse of Lemma \ref{lem:restrict} for subgroups of finite index in virtually polycyclic groups. The assumption on the group $\Gamma$ is used in order to have that every subgroup is finitely generated.

\begin{Lem}
\label{lem:subgroup}
Let $\Gamma$ be a virtually polycyclic group, $\varphi: \Gamma \to \Gamma$ a monomorphism and $\Gamma^\prime < \Gamma$ a finite index subgroup which is $\varphi$-invariant. The monomorphism $\varphi$ is strongly scale-invariant if and only if $\restr{\varphi}{\Gamma^\prime}$ is strongly scale-invariant. 
\end{Lem}

\begin{proof}
Since we are working in virtually polycyclic groups, the condition that the image has finite index is always satisfied. As mentioned before, one implication holds in general and is given by Lemma \ref{lem:restrict}. For the other implication, assume that $\restr{\varphi}{\Gamma^\prime}$ is strongly scale-invariant and assume that $H < \Gamma$ is a subgroup for which $\varphi(H) = H$. Since $\Gamma$ is virtually polycyclic, the subgroup $H$ is finitely generated. The intersection $H \cap \Gamma^\prime$ is a finite index subgroup of $H$, and hence some iterate $\varphi^k$ of $\varphi$ satisfies $\varphi^k(H \cap \Gamma^\prime) = H \cap \Gamma^\prime$. In particular, $H \cap \Gamma^\prime$ is finite because $\restr{\varphi^k}{\Gamma^\prime}$ is strongly scale-invariant and thus $H$ is finite as well.
\end{proof}

This lemma seems to imply that if $\Gamma$ contains a strongly scale-invariant subgroup of finite index, then the group itself is strongly scale-invariant. This is true by the main result in this paper, but it does not follow immediately from Lemma \ref{lem:subgroup}, since not every group morphism of a finite index subgroup can be extended to the entire group.

\subsection{Strongly scale-invariant virtually polycyclic groups}
\label{sec:maintheorem}

In this second part, we study strongly scale-invariant monomorphisms of virtually polycyclic groups via their extensions to the $\Q$-algebraic hull. We first shift our attention to nilpotent groups, for which we have that the rational points of the $\Q$-algebraic hull are equal to the  rational Mal'cev completion. In this part, $N$ will always be a torsion-free nilpotent group.

As explained before, every monomorphism $\varphi: N \to N$ of a torsion-free nilpotent group $N$ uniquely extends to an automorphism $\phi: N^\Q \to N^\Q$ on the rational Mal'cev completion $N^\Q$, where the eigenvalues of $\varphi$ and $\phi$ are identical, and checking whether $\varphi$ is strongly scale-invariant depends on its eigenvalues.
\begin{Prop}
	\label{prop:eigenvalue1}
	Let $N^\Q$ be a radicable torsion-free nilpotent group and $\phi: N^\Q \to N^\Q$ an automorphism.
	\begin{itemize}
		\item If $1$ is an eigenvalue of $\phi$, then there exists $x \in N^\Q$ such that $\phi(x) = x$.
		\item If $1$ is not an eigenvalue of $\phi$, then for every $x \in N^\Q$, there exists a unique $y$ in $N^\Q$ with $x = y \phi(y)^{-1}$.
	\end{itemize}
\end{Prop}

\begin{proof}
Write $\psi$ for the Lie algebra automorphism corresponding to $\phi$. For the first statement, if $0 \neq X \in \lien^\Q$ is an eigenvector for eigenvalue $1$, then $x = \exp(X) \in N^\Q$ will satisfy the condition, because $\phi(x) = \phi(\exp(X)) = \exp(\psi(X)) = \exp(X) = x$.

We prove the second statement by induction on the nilpotency class of $N^\Q$. If $N^\Q$ is abelian, it is isomorphic to $\Q^m$ for some $m > 0$ and $\psi =\phi$. Hence this follows directly from the fact that $1$ is not an eigenvalue of $\phi$ if and only if the map $(I - \phi)$ is invertible. This implies as well that the element $y$ is unique.  

Now assume that $c > 1$ and take any $x \in N^\Q$. By taking the quotient in $\faktor{N^\Q}{[N^\Q,N^\Q]}$, we know there exists a unique $y_1$ with $y_1 \phi(y_1)^{-1} [N^\Q,N^\Q] = x [N^\Q,N^\Q]$ from the abelian case. This means that $y_1^{-1} x \phi(y_1) \in [N^\Q,N^\Q]$. Applying the induction hypothesis on this element, we get that there exists $y_2 \in [N^\Q,N^\Q]$ with $y_1^{-1} x \phi(y_1)  = y_2 \phi(y_2)^{-1}$. The element $y_1 y_2$ satisfies the conditions of the proposition. For uniqueness, we trace back the argument in reverse order.
\end{proof}

\begin{Cor}
	\label{cor:discoeig}
Let $N$ be a torsion-free nilpotent group. If $\varphi: N \to N$ is a strongly scale-invariant monomorphism, then $1$ is not an eigenvalue of $\varphi$.
\end{Cor}

\begin{proof}
Assume for a contradiction that $\varphi$ has eigenvalue $1$. Consider the extension $\phi: N^\Q \to N^\Q$ to the rational Mal'cev closure of $N$, which is an automorphism of $N^\Q$. Take $x \in N^\Q$ with $\phi(x) = x$ as in Proposition \ref{prop:eigenvalue1}. There exists an integer $m > 0$ such that $x^m \in N$ since $N$ is a full subgroup of $N^\Q$. The subgroup $H$ generated by $x^m$ is infinite and satisfies $\varphi(H) = H$, which is a contradiction.
\end{proof}

The converse is not true, for example by considering the automorphism $\begin{pmatrix} 2 & 1 \\ 1 & 1\end{pmatrix}$ on $\Z^2$. This map does not have $1$ as eigenvalue, but it not strongly scale-invariant either as it an automorphism. The monomorphisms of nilpotent groups which are strongly scale-invariant can be characterized depending only on the eigenvalues, but this is not necessary for our purposes. We do need in the next section that if all eigenvalues of $\varphi$ are bigger than $1$ in absolute value, in which case we call $\varphi$ \emph{expanding}, that the map $\varphi$ is strongly scale-invariant.

\begin{Lem}
	\label{lem:expanding}
If $N$ is a finitely generated torsion-free nilpotent group and $\varphi: N \to N$ is expanding, then $\displaystyle \bigcap_{n > 0}\varphi^n(N)$ is trivial. In particular the monomorphism $\varphi$ is strongly scale-invariant.
\end{Lem}

\begin{proof}
We prove this via induction, where the abelian case is immediate by definition. For general nilpotent groups $N$, the quotient $\faktor{N}{\sqrt[N]{[N,N]}}$ is abelian and the induced map $\overline{\varphi}$ is again expanding. Hence $$\displaystyle \bigcap_{n > 0}\overline{\varphi}^n\left( \faktor{N}{\sqrt[N]{[N,N]}} \right) =\sqrt[N]{[N,N]}$$ or equivalently $\displaystyle \bigcap_{n > 0}\varphi^n(N) \subset \sqrt[N]{[N,N]}.$ Since $\sqrt[N]{[N,N]}$ is of nilpotency class strictly smaller than $N$ and the restriction of $\varphi$ to $\sqrt[N]{[N,N]}$ is still expanding, the induction hypothesis yields $\displaystyle \bigcap_{n > 0}\varphi^n(N) = \{e\}$ via Lemma \ref{lem:defwithH} and the fact that $N$ is torsion-free.
\end{proof}

%
%

For the proof, it will be crucial to consider the $\Q$-algebraic hull of a finite index subgroup. 

\begin{Lem}
	\label{lem:hullfi}
Let $\Gamma^\prime < \Gamma$ be a finite index subgroup of the virtually polycyclic group $\Gamma$ with $\Q$-algebraic hull $\Gamma^\Q$. The $\Q$-algebraic hull of $\Gamma^\prime$ is equal to the Zariski-closure of $\Gamma^\prime$ in $\Gamma^\Q$. In particular, if $\Gamma^\Q$ is connected, the $\Q$-algebraic hull of $\Gamma$ and $\Gamma^\prime$ coincide.
\end{Lem}

\begin{proof}
Write $G$ for the Zariski-closure of $\Gamma^\prime$, then we check the three condition of Definition \ref{def:hull}. The first condition is immediate by definition of $G$. The group $G$ is a finite index subgroup of $\Gamma^\Q$, implying that the unipotent radical of this group is equal to $U(\Gamma^\Q)$, yielding the second condition. Also the third condition follows since $G$ is a subgroup of $\Gamma^\Q$ and $U(\Gamma^\Q) = U(G)$.

For the last statement, note that $G$ is a closed subgroup of finite index in $\Gamma^\Q$ and hence also open, showing that $G = \Gamma^\Q$.
\end{proof}
We are now ready for the main result of this section.
\begin{Thm}\label{thm:main}
	Let $\Gamma$ be a virtually polycyclic group which is strongly scale-invariant, then $\Gamma$ is virtually nilpotent.
\end{Thm}

\begin{proof}
	First we look for suitable finite index subgroups of $\Gamma$ that are injectively characteristic and thus strongly scale-invariant by Lemma \ref{lem:subgroup}. It suffices to show that these finite index subgroups are virtually nilpotent, so each time we will replace $\Gamma$ by the finite index subgroup to achieve stronger assumptions. Note that every virtually polyclic group has a subgroup of finite index which is torsion-free and polycyclic, see \cite[Lemma 4.6.]{ragh72-1}. By replacing $\Gamma$ by $\Gamma^k$ for some $k > 0$, we can thus assume that $\Gamma$ is polycyclic and torsion-free. In particular $\Gamma$ has no finite normal subgroup and by Theorem \ref{thm:existQ}, $\Gamma$ has a $\Q$-algebraic hull $\Gamma^\Q$, with connected component $\left(\Gamma^\Q\right)^0$ of the identity element. The subgroup $\Gamma \cap \left(\Gamma^\Q\right)^0$ has finite index in $\Gamma$, with $\Q$-algebraic hull $\left(\Gamma^\Q\right)^0$ by Lemma \ref{lem:hullfi}. By again replacing $\Gamma$ by $\Gamma^k$ for some $k > 0$, we can hence assume that $\Gamma^\Q$ is connected. Finally, a similar argument shows that we can assume that $\faktor{\Gamma}{N}$ is torsion-free, with $N$ the Fitting subgroup of $\Gamma$. 

Let $\varphi: \Gamma \to \Gamma$ be a strongly scale-invariant monomorphism, with $\Phi$ the algebraic extension of $\varphi$ to $\Gamma^\Q$. 
The group $\faktor{\Gamma^\Q}{U(\Gamma^\Q)}$ is a torus and hence by the rigidity of tori, see \cite[8.10]{bore91-1}, some power of $\Phi$ induces the identity map on  $\faktor{\Gamma^\Q}{U(\Gamma^\Q)}$. So by taking some iterate of $\varphi$, which is strongly scale-invariant by Lemma \ref{lem:iterate}, we can assume that $\Phi$ is the identity on $\faktor{\Gamma^\Q}{U(\Gamma^\Q)}$. We know that $N = \Gamma \cap U(\Gamma^\Q)$ is the Fitting subgroup of $\Gamma$, hence the quotient $\faktor{\Gamma}{N}$ can be considered as a subgroup of $\faktor{\Gamma^\Q}{U(\Gamma^\Q)}$. The induced map by $\varphi$ on $\faktor{\Gamma}{N}$ is then the restriction of the map induced by $\Phi$ on $\faktor{\Gamma^\Q}{U(\Gamma^\Q)}$. We conclude that $\varphi$ induces the trivial map on $\faktor{\Gamma}{N}$.

We will show that in fact $\Gamma = N$ and thus that $\Gamma$ is nilpotent. Note that the rational Mal'cev completion of $N$ can be considered as a subgroup of $U(\Gamma^\Q)$, and thus the extension of $\varphi$ to $N^\Q$ is given by the restriction of $\Phi$ to $N^\Q$. Assume that $x \in \Gamma \setminus N$, then we know that $\varphi(x) = x y$ for some $y \in N$. Since $\restr{\varphi}{N}$ is strongly scale-invariant by Lemma \ref{lem:restrict}, we know by Proposition \ref{prop:eigenvalue1} that there exists $z \in N^\Q \subset U(\Gamma^\Q)$ with $y = z \Phi(z)^{-1}$. Hence the element $x z \in \Gamma^\Q$ satisfies $$\Phi(x z) = \Phi(x) \Phi(z) = xy \Phi(z)= x z.$$ Consider the group $\Gamma^\prime$ generated by $\Gamma$ and $z$, which contains $\Gamma$ as a subgroup of finite index. The map $\Phi$ induces a monomorphism on $\Gamma^\prime$, which is strongly scale-invariant by Lemma \ref{lem:subgroup}. Since the element $x z$ has infinite order and $\Phi(xz) = xz$, we get a contradiction, showing that such an $x \in \Gamma \setminus N$ does not exist. This implies $\Gamma = N$ and thus the group $\Gamma$ is nilpotent. 
\end{proof}

\subsection{Finite Reidemeister number}
\label{sec:reidemeister} 
As an application of the methods, we study the Reidemeister number for iterates of monomorphisms on virtually polycyclic groups. If $G$ is a group and $\varphi: G \to G$ is an endomorphism, we say that $x, y \in G$ are $\varphi$-conjugate if there exists $z \in G$ with $x = z y \varphi(z)^{-1}$. This is an equivalence relation on the elements of the group and we denote by $[x]_\varphi$ the equivalence class of $x \in G$. The Reidemeister number $R(\varphi) \in \N \cup \{\infty\}$ is defined as the number of equivalence classes, so $$R(\varphi) = \# \{[x]_\varphi \mid x \in G \}.$$ 
The Reidemeister number is an important invariant of endomorphisms, related to the study of fixed points as described in \cite{fh94-1}. For an endomorphism $\varphi: \Z^m \to \Z^m$ the Reidemeister number is finite if and only if $\varphi$ does not have eigenvalue $1$. Starting from an endomorphism $\varphi: G \to G$ with $R(\varphi^n) < \infty$ for all iterates $\varphi^n$, one can construct the Reidemeister zeta function as $$R_\varphi(z) = \exp{\left(\sum_{n=1}^\infty \frac{R(\varphi^n)}{n} z^n\right)}.$$ An active research line is to investigate when the Reidemeister zeta function is rational, see \cite{fh94-1}.

Recall the following lemma about the Reidemeister number, see \cite[Lemma 1.1.]{gw09-1}, where we only use the second part of the original statement for the special case of a finite group $G$.

\begin{Lem}
	\label{lem:daci}Let $\varphi: \Gamma \to \Gamma$ be an endomorphism of a group $\Gamma$ with $\varphi$-invariant normal subgroup $N$ and quotient group $G = \faktor{\Gamma}{N}$. Consider the following commutative diagram
	\[
\xymatrix{ 1 \ar[r] & N \ar[r]\ar[d]_{\varphi_N} & \Gamma \ar[r]\ar[d]_{\varphi} &
	G \ar[r] \ar[d]_{\overline{\varphi}} &1\\
	1 \ar[r] & N \ar[r] & \Gamma\ar[r] &G \ar[r] & 1 } \]
where $\overline{\varphi}$ and $\varphi_N$ are the induced and restricted endomorphism, respectively. The following statements about the Reidemeister number hold.
\begin{enumerate}
	\item  If $R(\varphi) < \infty$, then $R(\overline{\varphi}) < \infty$ as well.
	\item If $G$ is a finite group, then $R(\varphi) < \infty$ implies $R(\varphi_N) < \infty$.
\end{enumerate} 
\end{Lem}

As another application of our methods, we now show that a finite Reidemeister number for all iterates $\varphi^n$ implies that the group $\Gamma$ is virtually nilpotent.
\begin{Thm}\label{cor:felshtyn}
	Let $\varphi: \Gamma \to \Gamma$ be a monomorphism of a virtually polycyclic group. If $R(\varphi^n) < \infty$ for all integers $n > 0$ the group $\Gamma$ must be virtually nilpotent. In particular, the Reidemeister zeta function $R_\varphi(z)$ of $\varphi$ is rational if the group $\Gamma$ is additionally torsion-free.
	\end{Thm}
\begin{proof}
	Let $N$ be the Fitting subgroup of $\Gamma$ and write $G = \faktor{\Gamma}{N}$ for the quotient. By replacing $\Gamma$ by $\Gamma^k$ for some $k > 0$, as in the proof of Theorem \ref{thm:main}, we can assume that both $\Gamma$ and $G$ are torsion-free, that $G$ is abelian and that the $\Q$-algebraic hull of $\Gamma$ exists and is connected. Note that the restriction of $\varphi$ to this finite index normal subgroup still satisfies the conditions of the theorem by Lemma \ref{lem:daci}.
	
	Note that $N$ is injectively characteristic, since it is equal to $U(\Gamma^\Q) \cap \Gamma$. So for every $n > 0$ we can consider the following commuting diagram
	\[
	\xymatrix{ 1 \ar[r] & N \ar[r]\ar[d]_{\varphi^n_N} & \Gamma \ar[r]\ar[d]_{\varphi^n} &
		G \ar[r] \ar[d]_{\overline{\varphi}^n} &1\\
		1 \ar[r] & N \ar[r] & \Gamma\ar[r] &G \ar[r] & 1,} \]
	where $\varphi_N$ is the restriction of $\varphi$ to $N$ and $\overline{\varphi}$ is the induced map on the quotient $G$. Using Lemma \ref{lem:daci}, we know that $R(\overline{\varphi}^n) < \infty$ for all $n > 0$. Exactly as in the proof of Theorem \ref{thm:main}, we have that some iterate of $\varphi$ induces the identity map on $G$. Hence the group $G$ trivial, because otherwise some iterate $\overline{\varphi}^n$ would have eigenvalue $1$, leading to $R(\overline{\varphi}^n) = \infty$.
		
	The last statement follows from \cite[Theorem 7.8.]{fl13-2}, which shows that the Reidemeister zeta function is rational for monomorphisms on torsion-free virtually nilpotent groups.
\end{proof}

\section{Strongly scale-invariant virtually nilpotent groups}
\label{sec:virtuallynilpotent}

Although in the class of virtually polycyclic groups only the virtually nilpotent ones can be strongly scale-invariant, not every virtually nilpotent group does in fact have this property. Groups for which every monomorphism is automatically an automorphism are called co-Hopfian, and the first examples of co-Hopfian torsion-free nilpotent groups are constructed in \cite{bele03-1}. Note that a infinite co-Hopfian group can never be strongly scale-invariant.  A full characterization of the co-Hopfian torsion-free nilpotent groups was independently given in \cite{corn14-1,dd14-1}. In this section we characterize the strongly scale-invariant virtually nilpotent groups by adapting the methods of \cite{dd14-1}.

Let $\Gamma$ be a finitely generated virtually nilpotent group and $N \triangleleft \Gamma$ a normal subgroup of finite index which is both torsion-free and nilpotent. Such a subgroup $N$ always exists and, although it is not unique, the rational Mal'cev completion $N^\Q$ of $N$ does not depend on the choice of $N$. Indeed, if $N_1$ and $N_2$ are two different normal subgroups of finite index, then there intersection has $N_1 \cap N_2$ has finite index in both $N_1$ and $N_2$, which by Proposition \ref{prop:samemalcev} shows that the rational Mal'cev completions are isomorphic. We call the group $N^\Q$ the \emph{radicable nilpotent group associated to the virtually nilpotent group $\Gamma$}. Note that two virtually nilpotent groups $\Gamma_1$ and $\Gamma_2$ are abstractly commensurable if and only if the radicable nilpotent groups associated to $\Gamma_1$ and $\Gamma_2$ are isomorphic.

The main result of this section shows that being strongly scale-invariant only depends on the radicable nilpotent group $N^\Q$ associated to $\Gamma$.
\begin{Thm}
	\label{thm:virtuallynilpotent}
Let $\Gamma$ be a finitely generated virtually nilpotent group with associated radicable nilpotent group $N^\Q$. The group $\Gamma$ is strongly scale-invariant if and only if the Lie algebra corresponding to $N^\Q$ has a positive grading. Moreover, if $\Gamma$ is strongly scale-invariant, there exists a monomorphism $\varphi: \Gamma \to \Gamma$ such that $\displaystyle \bigcap_{n>0} \varphi^n(\Gamma)$ is the maximal finite normal subgroup of $\Gamma$.
\end{Thm}
\noindent The maximal finite normal subgroup $H$ of a virtually polycyclic group $\Gamma$ is preserved by every monomorphism $\varphi: \Gamma \to \Gamma$ by \cite[Proposition 2.7.]{corn15-1}, meaning that $\varphi(H) = H$. Hence the last part of the theorem shows that there exists a monomorphism realizing the smallest possible finite subgroup as the intersection of the images.

One part of this theorem, namely that $\Gamma$ being strongly scale-invariant implies that the Lie algebra corresponding to $N^\Q$ has a positive grading will follow directly from \cite{corn14-1}. The proof of the other implication consists of three different steps. In the first step, we show that the group $\Gamma$ can be embedded in a semi-direct product $N^\Q \rtimes_\rho F$ with $F$ finite and $\rho: F \to \Aut(N^\Q)$ a morphism which is not necessarily injective. The second step constructs automorphisms of the semi-direct product $N^\Q \rtimes_\rho F$ which induce strongly scale-invariant monomorphisms on $\Gamma$, similarly as in \cite{dd14-1}. In the final step we show that taking a certain conjugate subgroup of $\Gamma$ in $N^\Q \rtimes_\rho F$ yields the stronger last property of the theorem.

\subsection{Semi-direct product}First we show how to embed $\Gamma$ in a semi-direct product following the methods of \cite{deki96-1}. In the semi-direct $N^\Q \rtimes_\rho F$ we denote by $N^\Q$ and $F$ the natural subgroups corresponding to the first and second component.
\begin{Prop}
	\label{prop:semidirect}
Let $\Gamma$ be a finitely generated virtually nilpotent group with associated radicable group $N^\Q$. There exists an injective morphism $i: \Gamma \to N^\Q \rtimes_\rho F$ such that $i(\Gamma) \cap N^\Q$ is a full subgroup of $N^\Q$. Moreover the maximal finite normal subgroup of $i(\Gamma)$ is equal to $i(\Gamma) \cap \ker(\rho)$.
\end{Prop}

\noindent The condition that $i(\Gamma) \cap N^\Q$ is a full subgroup ensures that $N^\Q$ is the radicable nilpotent group associated to $\Gamma$.  In order to prove the last part of the proposition, we will need the centralizer of a full subgroup of $N^\Q$ in the group $N^\Q \rtimes_\rho F$, which is given by this lemma. 

\begin{Lem}
	\label{lem:central}
Consider the semi-direct product $G = N^\Q \rtimes_\rho F$ for some morphism $\rho: F \to \Aut(N^\Q)$ with $F$ a finite group. The centralizer of a full subgroup $N$ of $N^\Q$ in $G$ is equal to $$C_G(N) = \left\{(x,f) \suchthat x \in Z(N^\Q), f \in \ker(\rho) \right\}.$$
\end{Lem}
\begin{proof}
It is clear that elements $(x,f)$ with $x \in Z(N^\Q)$ and $f \in \ker(\rho)$ centralize $N$, so it suffices to show the reverse inclusion. For this, note that $(x,f) \in G$ centralizes $N$ if and only if \begin{align}
\label{eq:conjugate}
y = (x,f) y (x,f)^{-1} = x \rho(f)(y) x^{-1}
\end{align} for all $y \in N$. Since every element in $N^\Q$ is of the form $y^{\frac{1}{m}}$ for some $y \in N$ and an integer $m > 0$, we conclude that $(x,f)$ centralizes $N^\Q$ as well. By considering Equation (\ref{eq:conjugate}) in the quotient group $\faktor{N^\Q}{[N^\Q,N^\Q]}$ with induced automorphism $\overline{\rho(f)}$, we find that $\overline{\rho(f)}(y [N^\Q,N^\Q]) = y [N^\Q,N^\Q]$ for all $y \in N^\Q$. By Lemma \ref{lem:inducedmap}, we conclude that $\rho(f) = \I_{N^\Q}$ or thus $f \in \ker(\rho)$. In particular, Equation \ref{eq:conjugate} implies that $x \in Z(N^\Q)$, giving the second inclusion.
\end{proof}
\begin{proof}[Proof of Proposition \ref{prop:semidirect}]
Let $N$ be a torsion-free nilpotent normal subgroup of $\Gamma$. Consider the short exact sequence $$1 \to N \to \Gamma \to F \to 1,$$ where $F = \faktor{\Gamma}{N}$ is a finite group. Every automorphism of $N$ uniquely extends to an automorphism of $N^\Q$, hence exactly as in \cite[Page 35]{deki96-1} there exists a group $\Gamma^\Q$ and an injective morphism $j: \Gamma \to \Gamma^\Q$ which fit in the following commuting diagram:
\[
\xymatrix{ 1 \ar[r] & N \ar[r]\ar[d] & \Gamma \ar[r]\ar[d]_{j} &
	F \ar[r] \ar@{=}[d] &1\\
	1 \ar[r] & N^\Q \ar[r] & \Gamma^\Q \ar[r] &F \ar[r] & 1. } \]
 By the same argument as \cite[Lemma 3.1.2.]{deki96-1}, the bottom short exact sequence splits because the group $N^\Q$ is radicable, leading to an isomorphism $\Gamma^\Q \approx N^\Q \rtimes_\rho F$. Hence we find an embedding $i: \Gamma \to N^\Q \rtimes_\rho F$ satisfying $i(\Gamma) \cap N^\Q = i(N)$ which is a full subgroup.

For the last part of the proposition, note that $i(\Gamma) \cap \ker(\rho)$ is a finite normal subgroup of $i(\Gamma)$. It suffices to show that every finite normal subgroup $H$ of $i(\Gamma)$ lies in $\ker(\rho)$. Since $H$ is finite, there exists some integer $k > 0$ such that for every element $y$ which normalizes $H$, the element $y^k$ centralizes $H$. In particular, for every element $y \in N$, we have that $i(y)^k$ centralizes $H$. Because the elements $i(y)^k$ generate a full subgroup of $N^\Q$, Lemma \ref{lem:central} implies that every element of $H$ is of the form $(x,f)$ with $f \in \ker(\rho)$ and $x \in Z(N^\Q)$. Since $H$ is finite and $Z(N^\Q)$ is torsion-free, this yields that every element of $H$ lies in $\ker(\rho)$, as we needed to show.
\end{proof}

Although the group $N^\Q$ is uniquely determined by the group $\Gamma$, both the group $F$ and the morphism $\rho: F \to \Aut(N^\Q)$ depend on the choice of the subgroup $N$. We give a concrete example illustrating this fact.
\begin{Ex}
	\label{ex:easiest}
Consider the group $\Gamma = \Z \rtimes \Z_2$, where $\Z_2 = \{\pm 1\}$ acts on $\Z$ by multiplication. We consider two different abelian normal subgroups, namely $N_1 = \Z \triangleleft \Gamma$ and $N_2 = 2 \Z = \Gamma^2 \triangleleft \Gamma$. In the first case, we have $F_1 = \faktor{\Gamma}{N_1} \approx \Z_2$ and the map 
\begin{align*}
i_1: \Gamma &\to \Q \rtimes_{\rho_1} \Z_2 \\
(z,t) & \mapsto (z,t)
\end{align*} with $\rho_1: \Z_2 \to \GL(\Q)$ the natural inclusion. In the second case, we have the finite group $F_2 =  \faktor{\Gamma}{N_2} \approx \Z_2 \oplus \Z_2$ with representation $\rho_2: \Z_2 \oplus \Z_2 \to \GL(\Q)$ given by $\rho_2(t_1,t_2) = t_1$. In this case the map $i_2: \Gamma \to \Q \rtimes_{\rho_2} F_2$ is given by 
\begin{align*}
i_2(z,t) = 
\begin{cases}
(z,t,1) & \text{if } z \in 2\Z \\
(z,t,-1) & \text{if } z \notin 2 \Z. \\
\end{cases}
\end{align*} 
Since $i_2(\Gamma) \cap \ker(\rho_2) = \{e\}$, also the embedding $i_2$ shows that $\Gamma$ has no finite normal subgroup.
	\end{Ex}

\subsection{Constructing the monomorphism}
The next proposition generalizes \cite[Theorem 4.2.]{dd14-1} which only considered torsion-free groups. The construction is almost identical, but it leads to monomorphisms preserving the maximal finite normal subgroup, which is a subgroup of $F$.
\begin{Prop}
	\label{prop:autoonsemi}
Let $\Gamma$ be a finitely generated virtually nilpotent group realized as a subgroup of $N^\Q \rtimes_\rho F$ with $N^\Q \cap \Gamma$ a full subgroup of $N^\Q$. If $N^\Q$ has a positive grading, there exists an automorphism $\Phi: N^\Q \rtimes F \to N^\Q \rtimes F$ leaving $\Gamma$ invariant and such that $$\bigcap_{n > 0} \Phi^n(\Gamma) = \Gamma \cap F.$$ 
\end{Prop}

\begin{proof}
Since $\rho(F)$ is a finite subgroup of $\Aut(N^\Q)$, there exists a positive grading which is preserved by $\rho(F)$, see \cite[Theorem 2.2.]{dere14-1}. In \cite[Corollary 3.3.]{dd14-1} it is shown that for every prime $p>0$, there exist expanding automorphisms $\phi_p \in \Aut(N^\Q)$ with determinant $p^m$ for some fixed $m >0$, which leave a full subgroup $N_0$ of $N^\Q$ invariant and commute with every element of $\rho(F)$. Note that each of these automorphisms also induce automorphisms $\Phi_p: N^\Q \rtimes_\rho F \to N^\Q \rtimes_\rho F$ by $$\Phi_p(x,f) = (\phi_p(x),f).$$ We claim that $\Phi = \Phi^k_p$ for some $p$ and $k$ satisfies the conditions of the proposition. 

Consider the full subgroup $N = \Gamma \cap N^\Q$. Let $N_1$ be the full subgroup generated by all $x \in N^\Q$ for which there exists $f \in F$ with $(x,f) \in \Gamma$. Note that $N < N_1$ is a subgroup of finite index. Take any normal subgroup $N_2 \triangleleft N_1$ such that $N_2 < N$. Exactly as in the proof of \cite[Theorem 4.2.]{dd14-1} there exists some prime $p$ and $k > 0$ such that $\phi_p^k(N_1) < N_1, \phi_p^k(N_2) < N_2, \phi_p^k(N) < N$ and such that $\phi_p^k$ induces the identity map on the finite group $\faktor{N_1}{N_2}$. We claim that the corresponding map $\Phi = \Phi_p^k$ satisfies the conditions of the proposition.

First we show that $\Gamma$ is $\Phi$-invariant. Take any $(x,f) \in \Gamma$ with $x \in N_1$ and $f \in F$. Then $$\Phi_p^k(x,f) = (\phi_p^k(x),f) = (x_2 x,f) = (x_2,e) (x,f)$$ where $x_2 \in N_2 <N$ and hence $\Phi_p^k(x,f) \in \Gamma$ for all $\gamma = (x,f) \in \Gamma$. Moreover, if we look at $\Phi^n(\Gamma) \subset \Phi^n(N_1) \rtimes \Phi^n(F) = \Phi^n(N_1) \rtimes F$, hence \begin{align*}\bigcap_{n > 0} \Phi^n(\Gamma) <F
\end{align*} by Lemma \ref{lem:expanding}. Since $\Gamma \cap F$ is automatically contained in $\Phi^n(\Gamma)$ for all $n > 0$, this shows that also the second condition is satisfied.
\end{proof} 

Note that the monomorphism constructed in Proposition \ref{prop:autoonsemi} is always strongly scale-invariant, but the intersection of the images can be bigger than the maximal finite normal subgroup.

\begin{Ex}
	\label{ex:finitesubgroup}
Let $\Gamma = \Z \rtimes \Z_2$ as in Example \ref{ex:easiest}. The monomorphisms constructed in Proposition \ref{prop:autoonsemi} are of the form 
\begin{align*}
\varphi_m: \Gamma &\to \Gamma\\ (x,t) &\mapsto (mx,t)
\end{align*}
with $m \in \N_0$. Note that $\varphi_m(\Z_2) = \Z_2$, although it is not a finite normal subgroup of $\Gamma$.
\end{Ex}

\subsection{Conjugate subgroup}
In order to get a group morphism such that $\displaystyle \bigcap_{n > 0} \varphi^n(\Gamma)$ is the maximal finite normal subgroup, we have to alter the embedding $\Gamma < N^\Q \rtimes_\rho F$. The following lemma shows that an embedding exists such that the finite normal subgroup is equal to the intersection with $F$.

\begin{Lem}
\label{lem:conjugate}
Let $\Gamma$ be a finitely generated virtually nilpotent group realized as a subgroup of $N^\Q \rtimes_\rho F$ with $N^\Q \cap \Gamma$ a full subgroup. There exists $x \in N^\Q$ such that $x \Gamma x^{-1} \cap F <\ker{\rho}$. 
\end{Lem}

In particular, the finite normal subgroup of the group $x \Gamma x^{-1}$ is exactly its intersection with $F$.

\begin{proof}
Every $\gamma \in \Gamma$ can be written uniquely as $(x_\gamma,f_\gamma)$ with $x_\gamma \in N^\Q$ and $f_\gamma \in F$. Take $N$ the full subgroup of $N^\Q$ generated by the elements $x_\gamma$ for all $\gamma \in \Gamma$. For every $f_\gamma \notin \ker(\rho)$, we have that also the induced map by $\rho(f_\gamma)$ on $\faktor{N^\Q}{[N^\Q,N^\Q]}$ is non-trivial by Lemma \ref{lem:inducedmap}. This means that the eigenspace for eigenvalue $1$ has dimension strictly smaller than the dimension of $\faktor{N^\Q}{[N^\Q,N^\Q]}$. Since we only have a finite number of such elements $f_\gamma$, each having an eigenspace for eigenvalue $1$ with dimension strictly smaller than the rational vector space $\faktor{N^\Q}{[N^\Q,N^\Q]}$, there exists $y \in N^\Q$ such that $y$ not an eigenvector for eigenvalue $1$ for $f_\gamma \notin \ker(\rho)$, so $y^{-1} \rho(f_\gamma)(y) \notin [N^\Q,N^\Q]$. The subgroup $\faktor{N}{[N^\Q,N^\Q]}$ is finitely generated, hence by taking $x = y^{\frac{1}{m}}$ for some integer $m > 0$, we can assume that $x$ satisfies $$x^{-1} \rho(f_\gamma)(x) [N^\Q,N^\Q] = \frac{y^{-1} \rho(f_\gamma)(y)}{m} [N^\Q,N^\Q]\notin \faktor{N}{[N^\Q,N^\Q]}.$$ We claim that $x$ satisfies the condition of the lemma.

Indeed, take any element $\gamma = (x_\gamma, f_\gamma) \in \Gamma$ with $f_\gamma \notin \ker(\rho)$. Then after conjugating, we get $$x \gamma x^{-1} = (x x_\gamma \rho(f_\gamma)(x^{-1}), f_\gamma)$$ and since $$x x_\gamma \rho(f_\gamma)(x^{-1}) [N^\Q,N^\Q] = x \rho(f_\gamma)(x^{-1}) x_\gamma [N^\Q,N^\Q] \neq [N^\Q,N^\Q]$$ by the last condition on $x$, the result follows.
\end{proof}
Combining all the previous steps, we are ready to prove Theorem \ref{thm:virtuallynilpotent}.
\begin{proof}[Proof of Theorem \ref{thm:virtuallynilpotent}]
First assume that $\Gamma$ is strongly scale-invariant and take any injectively characteristic subgroup $N$ of finite index which is nilpotent and torsion-free. Since $N$ is injectively characteristic, it is also strongly scale-invariant by Lemma \ref{lem:restrict}. In \cite[Theorem 1.11.]{corn14-1} it was shown that the corresponding rational Mal'cev completion has a positive grading.

For the other implication, consider $\Gamma$ as a subgroup of $N^\Q \rtimes_\rho F$ as in Proposition \ref{prop:semidirect}. By conjugating $\Gamma$ as in Lemma \ref{lem:conjugate}, we can assume that $\Gamma \cap F = \Gamma \cap \ker(\rho)$, which is moreover the maximal finite normal subgroup by Proposition \ref{prop:semidirect}. The automorphism constructed in Proposition \ref{prop:autoonsemi} induces a monomorphism $\varphi$ on $\Gamma$ which is strongly scale-invariant since $F$ is finite. Moreover, because of our assumption we have that $\displaystyle \bigcap_{n>0} \varphi^n(\Gamma) = \Gamma \cap F = \Gamma \cap \ker{(\rho)}$ is the maximal finite normal subgroup of $\Gamma$.
\end{proof}

We illustrate the construction in Theorem \ref{thm:virtuallynilpotent} via Example \ref{ex:finitesubgroup}.

\begin{Ex}
Consider again the group $\Gamma = \Z \rtimes \Z_2$, as a subgroup of $\Q \rtimes \Z_2$ as given by $i_1$ in Example \ref{ex:easiest}. Taking $x = \frac{1}{4}$, we get
\begin{align*} 
\Gamma^\prime = x \Gamma x^{-1} &= \left\{\left(\frac{1}{4},1\right) (z,t) \left(-\frac{1}{4},1\right) \suchthat z \in \Z, t \in \{\pm 1\}  \right\} \\ 
&= \left\{ \left(z + \frac{1}{2},-1\right)  \suchthat z \in \Z \right\} \cup  \left\{ \left(z,1\right)\suchthat z \in \Z \right\} 
\end{align*}
and thus $\Gamma^\prime \cap \Z_2$ is trivial.

For $m=3$ in Example \ref{ex:finitesubgroup} we get the automorphism $\Phi: \Q \rtimes \Z_2 \to \Q \rtimes \Z_2$ defined as $\Phi(t,z) = (3t, z)$ which induces a strongly scale-invariant monomorphism on $\Gamma^\prime$. On the group $\Gamma = \Z \rtimes \Z_2$ with generators $a$ and $b$ for $\Z$ and $\Z_2$ respectively, this induced map $\varphi: \Gamma \to \Gamma$ is given by $\varphi(a) = a^3$ and $\varphi(b) = ab$.

\end{Ex}

The virtually nilpotent groups $\Gamma$ for which every finite normal subgroup is trivial are called almost-crystallographic, see \cite{deki96-1}. They can be represented as isometries of a simply connected nilpotent Lie group $N^\R$, and we say that the group is modeled on the Lie group $N^\R$. Note that $N^\R$ is exactly the $\Q$-algebraic hull of any torsion-free nilpotent subgroup of $\Gamma$, and hence the radicable group $N^\Q$ associated to $\Gamma$ is equal to rational points of $N^\R$. The existence of a positive grading on $N^\Q$ is equivalent to the existence of a positive grading on the real Lie algebra corresponding to $N^\R$, see \cite[Theorem 1.4.]{corn14-1} or \cite[Theorem 1.2.]{dere14-1}. As an immediate consequence of Theorem \ref{thm:virtuallynilpotent}, we have the following result.

\begin{Cor}
	\label{cor:crystal}
Let $\Gamma$ be a finitely generated virtually nilpotent group. There exists a monomorphism $\varphi:\Gamma \to \Gamma$ with $\displaystyle \bigcap_{n > 0}\varphi^n(\Gamma)$ trivial if and only if $\Gamma$ is almost crystallographic and modeled on a Lie group $N^\R$ with a positive grading.
\end{Cor}

Note that the monomorphisms constructed in Corollary \ref{cor:crystal} following the proof of Theorem \ref{thm:virtuallynilpotent} satisfy the stronger property that they are expanding. This means that, after chosing a finite generating set $S$ for the group $\Gamma$ with corresponding word metric $\Vert \cdot \Vert_S$, it holds that $\Vert \varphi^k(\gamma) \Vert_S \geq c \lambda^k \Vert \gamma \Vert_S$ for some $c > 0, \lambda > 1$. These expanding monomorphisms for which the image has finite index in the group exist only on virtually nilpotent groups by the work of J.~Franks and M.~Gromov, see \cite{fran70-1,grom81-1}. Since every expanding map is automatically strongly scale-invariant, Theorem \ref{thm:virtuallynilpotent} can also be considered as a characterization of the groups admitting an expanding monomorphism with finite index image. More details can be found in \cite{dd14-1}. 

\bibliography{ref}

\begin{thebibliography}{10}

\bibitem{bg06-1}
Oliver Baues and Fritz Grunewald.
\newblock Automorphism groups of polycyclic-by-finite groups and arithmetic
  groups.
\newblock {\em Publ. Math. Inst. Hautes \'{E}tudes Sci.}, (104):213--268, 2006.

\bibitem{bele03-1}
Igor Belegradek.
\newblock On co-{H}opfian nilpotent groups.
\newblock {\em Bull. London Math. Soc.}, 35(6):805--811, 2003.

\bibitem{bore91-1}
Armand Borel.
\newblock {\em Linear algebraic groups}, volume 126 of {\em Graduate Texts in
  Mathematics}.
\newblock Springer-Verlag, second edition, 1991.

\bibitem{corn15-1}
Yves Cornulier.
\newblock Commability and focal locally compact groups.
\newblock {\em Indiana Univ. Math. J.}, 64(1):115--150, 2015.

\bibitem{corn14-1}
Yves Cornulier.
\newblock Gradings on {L}ie algebras, systolic growth, and cohopfian properties
  of nilpotent groups.
\newblock {\em Bull. Soc. Math. France}, 144(4):693--744, 2016.

\bibitem{deki96-1}
Karel Dekimpe.
\newblock {\em {A}lmost-{B}ieberbach {G}roups: {A}ffine and Polynomial
  Structures}, volume 1639 of {\em Lect. Notes in Math.}
\newblock Springer--Verlag, 1996.

\bibitem{dd14-1}
Karel Dekimpe and Jonas Der\'{e}.
\newblock Expanding maps and non-trivial self-covers on infra-nilmanifolds.
\newblock {\em Topol. Methods Nonlinear Anal.}, 47(1):347--368, 2016.

\bibitem{dere14-1}
Jonas Der\'{e}.
\newblock Gradings on {L}ie algebras with applications to infra-nilmanifolds.
\newblock {\em Groups Geom. Dyn.}, 11(1):105--120, 2017.

\bibitem{dere18-1}
Jonas Der\'{e}.
\newblock Nil-affine crystallographic actions of virtually polycyclic groups.
\newblock Preprint.
\newblock ArXiv:1810.11290.

\bibitem{dl57-1}
J.~Dixmier and W.G. Lister.
\newblock Derivations of nilpotent {L}ie algebras.
\newblock {\em Proc. Am. Math. Soc.}, 8:155--158, 1957.

\bibitem{fh94-1}
Alexander Fel{'}shtyn and Richard Hill.
\newblock The {R}eidemeister zeta function with applications to {N}ielsen
  theory and a connection with {R}eidemeister torsion.
\newblock {\em $K$-Theory}, 8(4):367--393, 1994.

\bibitem{fl13-2}
Alexander Fel'shtyn and Jong~Bum Lee.
\newblock The {N}ielsen and {R}eidemeister numbers of maps on
  infra-solvmanifolds of type ({R}).
\newblock {\em Topology Appl.}, 181:62--103, 2015.

\bibitem{fran70-1}
J.~Franks.
\newblock Anosov diffeomorphisms.
\newblock {\em Global Analysis: Proceedings of the Symposia in Pure
  Mathematics}, 14,:pp. 61--93, 1970.

\bibitem{gw09-1}
Daciberg Gon\c{c}alves and Peter Wong.
\newblock Twisted conjugacy classes in nilpotent groups.
\newblock {\em J. Reine Angew. Math.}, 633:11--27, 2009.

\bibitem{grom81-1}
M.~Gromov.
\newblock Groups of polynomial growth and expanding maps.
\newblock {\em Institut des Hautes \'Etudes Scientifiques}, (53):pp. 53--73,
  1981.

\bibitem{hlv20-1}
Steven Hurder, Olga Lukina, and Wouter Van~Limbeek.
\newblock Cantor dynamics of renormalizable groups.
\newblock {\em Preprint, \url{https://arxiv.org/abs/2002.01565}}, 2020.

\bibitem{np11-1}
Volodymyr Nekrashevych and G\'{a}bor Pete.
\newblock Scale-invariant groups.
\newblock {\em Groups Geom. Dyn.}, 5(1):139--167, 2011.

\bibitem{ragh72-1}
M.~S. Raghunathan.
\newblock {\em Discrete {S}ubgroups of {L}ie {G}roups}, volume~68 of {\em
  Ergebnisse der Mathematik und ihrer Grenzgebiete}.
\newblock Springer-{V}erlag, 1972.

\bibitem{sega83-1}
Daniel Segal.
\newblock {\em Polycyclic {G}roups}.
\newblock Cambridge University Press, 1983.

\bibitem{shub69-1}
M.~Shub.
\newblock Endomorphisms of compact differentiable manifolds.
\newblock {\em Amer. J. Math.}, 91,:pp. 175--199, 1969.

\bibitem{vanl18-1}
Wouter van Limbeek.
\newblock Towers of regular self-covers and linear endomorphisms of tori.
\newblock {\em Geom. Topol.}, 22(4):2427--2464, 2018.

\end{thebibliography}
\bibliographystyle{plain}

\end{document}